\documentclass[11pt, a4paper]{amsart}
\usepackage{amsfonts,amssymb,amsmath,amsthm,amscd,mathtools,multicol,tikz, tikz-cd,caption,enumerate,mathrsfs,thmtools,cite}
\usepackage{inputenc}
\usepackage[foot]{amsaddr}
\usepackage[pagebackref=true, colorlinks, linkcolor=blue, citecolor=red]{hyperref}
\usepackage{latexsym}
\usepackage{fullpage}
\usepackage{microtype}
\usepackage{subfiles} 

\usepackage{palatino}
\makeatletter
\makeindex 
\newcommand{\be}{\begin{equation}}
\newcommand{\ee}{\end{equation}}
\newcommand{\beano}{\begin{eqn*}} 
	\newcommand{\eeano}{\end{eqnarray*}}
\newcommand{\ba}{\begin{array}}
	\newcommand{\ea}{\end{array}}

\declaretheoremstyle[headfont=\normalfont]{normalhead}

\newtheorem{theorem}{Theorem}[section]

\newtheorem{lemma}[theorem]{Lemma}
\newtheorem{corollary}[theorem]{Corollary}

\newtheorem{proposition}[theorem]{Proposition}

\theoremstyle{definition}

\newtheorem{example}[theorem]{Example}

\newcommand{\rd}{\mathrm{rd}}
\newcommand{\ed}{\operatorname{ed}}
\newcommand{\rank}{\operatorname{rank}}

\newcommand{\Sym}{\operatorname{S}}
\newcommand{\Alt}{\operatorname{A}}
\newcommand{\GL}{\operatorname{GL}}

\newcommand{\PSL}{\operatorname{PSL}}
\newcommand{\SL}{\operatorname{SL}}

\numberwithin{equation}{section}

\begin{document}
\title{Essential Dimension of Small Finite Groups}
\author{Dilpreet Kaur}
\email{dilpreetkaur@iitj.ac.in}
\address{Department of Mathematics, Indian Institute of Technology Jodhpur, INDIA}
\thanks{Dilpreet Kaur was supported by International Mobility Research Grant IRMG/DPK/20220036 from Indian Institute of Technology Jodhpur, India}

\author{Zinovy Reichstein}
\email{reichst@math.ubc.ca}
\address{Department of Mathematics, University of British Columbia, Vancouver, CANADA}
\thanks{Zinovy Reichstein was partially supported by an Individual Discovery Grant RGPIN-2023-03353 from the
	National Sciences and Engineering Research Council of
	Canada.}


\begin{abstract} We compute the essential dimension of finite groups of order $\leqslant 63$.
\end{abstract}

\maketitle
\section{Introduction}

Let $G$ be a finite group. The {\em representation dimension} $\rd(G)$ is the minimal dimension of a faithful complex 
linear representation of $G$, i.e., the smallest positive integer $r$ such that $G$ is isomorphic to 
a subgroup of $\GL_r(\mathbb C)$. The {\em essential dimension} $\ed(G)$ is the minimal dimension of a faithful linearizable $G$-variety. Here by a faithful $G$-variety we mean an algebraic complex variety $X$ with a faithful action of $G$.
We say that $X$ is linearizable if there exists a $G$-equivariant dominant rational map $V \dasharrow X$, 
where $V$ is a finite-dimensional complex vector space with a linear action for $G$. 

Both the representation dimension $\rd(G)$ and the essential dimension $\ed(G)$ are interesting numerical invariants 
of $G$. Of the two, the representation dimension is more accessible. Indeed, if one knows the irreducible representations of $G$, then one can figure out how to put them together
to obtain a faithful representation of minimal dimension, at least in principle. In practice this is not always straightforward and $\rd(G)$ has been the subject of several 
recent papers~\cite{mr2, bardestani, bardestani2, moreto, knight}. For small groups, representation dimension can be easily calculated using computer algebra software GAP. The code can be found in ~\cite{KKS23}.

It is clear from the definitions of representation dimension and essential dimension that 
\begin{equation} \label{e.ed-rd} \ed(G) \leqslant \rd(G). 
\end{equation}
If $G$ is abelian, then equality holds; see Lemma~\ref{lem.first}(b). If $G$ is a $p$-group, then equality also holds by the Karpenko-Merkurjev Theorem; see Theorem~\ref{p-groups}.
For other finite groups $G$, the problem of computing $\ed(G)$ is largely open.
For example, the exact value of the essential dimension $\ed(\Sym_n)$ of the symmetric group $\Sym_n$ is only known for $n \leqslant 8$; see~\cite[Section 3i]{merkurjev-survey}. The essential dimension of the semidirect product of two cyclic groups of relatively prime order is only known in a handful of small special cases; cf.~\cite[Corollary 4.4, Remarks 4.5 and 4.6]{brv-mixed}.

The purpose of this paper is to compute $\ed(G)$ for finite groups $G$ 
of small order, up to $63$. The resulting values are tabulated in 
Sections~\ref{sect.<32}-\ref{sect.56-63}. Note that we only considered non-abelian groups, since
for abelian groups the situation is very simple; see~Lemma~\ref{lem.first}(b).
Sections~\ref{sect.prel}-\ref{sect.prel48} contain the background material used to justify the values for the essential dimension in our tables.

As we mentioned above, the problem of computing essential dimension of an arbitrary finite group is wide open. One of the questions motivating this paper was to find out at what point 
the existing methods break down. It appears that this happens for the first time when the order of the group is $54$. In fact, there are three groups in our tables, with GAP Ids $(54, 5)$, $(55, 1)$ and $(56, 11)$, where 
we have not been able to determine the exact value of the essential dimension. We give upper and lower bounds in these three cases.

Another motivating question was to investigate the ratio $\displaystyle \frac{\ed(G)}{\rd(G)}$. 
As mentioned above, this ratio is $\leqslant 1$ in general. Equality holds when $G$ is either an abelian group or a $p$-group. If $G$ ranges over all finite groups, this ratio can be arbitrarily small; see Corollary~\ref{cor.ratio}.
Let $\alpha(n)$ be the minimal value of $\displaystyle 
\frac{\ed(G)}{\rd(G)}$, as $G$ ranges over all finite groups of order $\leqslant n$.
It follows from~\cite[Theorem 2]{reichstein-jordan} that
$\displaystyle \alpha(n) \geqslant \frac{1}{j(n)}$, where $j(n)$ is the uniform Jordan constant
for the groups of birational automorphism $\operatorname{Bir}(X)$, as $X$ ranges over all complex rationally connected 
$n$-dimensional varieties. The existence of $j(n)$ is due to Birkar~\cite[Corollary 1.3]{birkar}. 
The precise value of $j(n)$ is not known even for small $n$
(see Prokhorov and Shramov~\cite{jordan-constant}); for general $n$ it appears to be far out of reach.
Moreover, even if $j(n)$ were known for all $n$,
it is likely that $\alpha(n)$ is much larger than $\displaystyle \frac{1}{j(n)}$. 
 
Our tables show that $\alpha(63)$ is $\displaystyle \frac{1}{3}$. This value is attained by the group $C_7 \rtimes C_6$ with GAP Id $(42, 1)$. 
Note that this group is constructed following the recipe in the proof of Corollary~\ref{cor.ratio} with $r = 2$ 
and $q = 7$.

\section{Notational conventions}
\label{sect.notation}

\begin{center}
\begin{tabular}{p{1.5cm}p{0.5 cm}p{11cm}}
  $C_n$   & : &  cyclic group of order $n.$ \\
  $D_n$   & : &  dihedral group of order $n.$\\
  $Q_n$   & : &  quaternion  group of order $n.$\\
  $QD_n$   & : & quasi-dihedral group of order $n.$\\
  $\Sym_n$   & : &  symmetric group of degree $n.$\\
  $\Alt_n$ & : & alternating group of degree $n.$\\
  $\GL_n(\mathbb{F}_q)$ & : &  general linear group of $n\times n$ matrices over the field containing $q$ elements.\\
 $\SL_n(\mathbb{F}_q)$ & : &  special linear group of $n\times n$ matrices over the field containing $q$ elements.\\
 $\PSL_n(\mathbb{F}_q)$ & : &  projective special linear group of $n\times n$ matrices over the field containing $q$ elements. \\
   $G \cdot H$ & : & extension of $G$ by $H$. \\
 $G\rtimes H$ & : & semidirect product of groups $G$ and $H$. \\
  $G \times H$ & : & direct product of groups $G$ and $H$.
\end{tabular}
\end{center}

\newpage
\section{Non-abelian groups of order $\leqslant 31$}
\label{sect.<32}

\begin{center}
\begin{tabular}{ |p{1.5cm}|p{4cm}|p{0.5 cm}|p{0.5 cm}|p{7cm}| } 
 \hline
 GAP Id & Structure Description & RD & ED & Explanation  \\ 
 \hline\hline
 (6,1) & $D_6$ & 2 & 1 & Theorem~\ref{ED=1}(a) \\ 
 \hline
 (8,3) & $D_8$ & 2 & 2 & Theorem~\ref{ED=1}(b) \\ 
 \hline
 (8,4) & $Q_8$ & 2 & 2 & Theorem~\ref{ED=1}(b) \\ 
 \hline
 (10,1) & $D_{10}$ & 2 & 1 & Theorem~\ref{ED=1}(a) \\ 
 \hline
 (12,1) & $C_3 \rtimes C_4$ & 2 & 2 & Theorem~\ref{ED=1}(a)  \\ 
 \hline
 (12,3) & $\operatorname{A}_4$ & 3 & 2 & Subgroup of $\operatorname{S}_5$ and Theorem~\ref{ED=2}(vii)\\ 
 \hline
 (12,4) & $D_{12}$ & 2 & 2 & Theorem~\ref{ED=1}(b) \\ 
 \hline
 (14,1) & $D_{14}$ & 2 & 1 & Theorem~\ref{ED=1}(a) \\ 
 \hline
 (16,3) & $(C_4 \times C_2) \rtimes C_2$ & 3 & 3 &  Theorem~\ref{p-groups}\\ 
 \hline
 (16,4) & $C_4 \rtimes C_4$ & 3 & 3 &  Theorem~\ref{p-groups}\\ 
 \hline
 (16,6) & $C_8 \rtimes C_2$ & 2 & 2 &  Theorem~\ref{ED=1}(b)\\ 
 \hline
 (16,7) & $D_{16}$ & 2 & 2 & Theorem~\ref{ED=1}(b) \\ 
 \hline
 (16,8) & $QD_{16}$ & 2 & 2 & Theorem~\ref{ED=1}(b) \\ 
 \hline
 (16,9) & $Q_{16}$ & 2 & 2 & Theorem~\ref{ED=1}(b) \\ 
 \hline
  (16,11) & $C_2 \times D_8$ & 3 & 3 &  Theorem~\ref{p-groups} \\ 
 \hline
 (16,12) & $C_2 \times Q_8$ & 3 & 3 & Theorem~\ref{p-groups}\\ 
 \hline
 (16,13) & $(C_4 \times C_2) \rtimes C_2$ & 2 & 2 & Theorem~\ref{ED=1}(b) \\ 
 \hline
 (18,1) & $D_{18}$ & 2 & 1 & Theorem~\ref{ED=1}(a) \\ 
 \hline
 (18,3) & $C_3 \times D_6$ & 2 & 2 & Theorem~\ref{ED=1}(b) \\ 
 \hline
 (18,4) & $(C_3 \times C_3)\rtimes C_2$ & 4 & 2 & Subgroup of $(C_3 \times C_3) \rtimes \mathcal{G}_2$ and Theorem~\ref{ED=2}(iii)\\ 
 \hline
 (20,1) &  $C_5 \rtimes C_4$ & 2 & 2 & Theorem~\ref{ED=1}(b) \\ 
 \hline
 (20,3) & $C_5 \rtimes C_4$ & 4 & 2 & Subgroup of $\operatorname{S}_5$ and Theorem~\ref{ED=2}(vii)\\ 
 \hline
 (20,4) & $D_{20}$ & 2 & 2 & Theorem~\ref{ED=1}(b) \\ 
 \hline
 (21,1) & $C_7 \rtimes C_3$ & 3 & 2 & Subgroup of $\operatorname{PSL}_2(\mathbb{F}_7)$ and Theorem~\ref{ED=2}(vi)\\ 
 \hline
 (22,1) & $D_{22}$ & 2 & 1 & Theorem~\ref{ED=1}(a) \\ 
 \hline
 (24,1) & $C_3 \rtimes C_8$ & 2 & 2 & Theorem~\ref{ED=1}(b) \\ 
 \hline
 (24,3) & $\operatorname{SL}_2(\mathbb{F}_3)$ & 2 & 2 & Theorem~\ref{ED=1}(b) \\ 
 \hline
 (24,4) & $C_3 \rtimes Q_8$ & 2 & 2 & Theorem~\ref{ED=1}(b) \\ 
 \hline
 (24,5) & $C_4 \times \operatorname{S}_3$ & 2 & 2 & Theorem~\ref{ED=1}(b) \\ 
 \hline
 (24,6) & $D_{24}$ & 2 & 2 & Theorem~\ref{ED=1}(b) \\ 
 \hline
 (24,7) & $C_2 \times (C_3 \rtimes C_4)$ & 3 & 3 &  Center is non-trivial and Prop. \ref{Prop.non-trivial-center} \\ 
 \hline
 (24,8) & $(C_6 \times C_2) \rtimes C_2$ & 2 & 2 & Theorem~\ref{ED=1}(b) \\ 
 \hline
 (24,10) & $C_3\times D_8$ & 2 & 2 & Theorem~\ref{ED=1}(b) \\ 
 \hline
 (24,11) & $C_3 \times Q_8$ & 2 & 2 & Theorem~\ref{ED=1}(b) \\ 
 \hline
 (24,12) & $\operatorname{S}_4$ & 3 & 2 &  Subgroup of $\operatorname{S}_5$ and Theorem~\ref{ED=2}(vii)\\
 \hline
 (24,13) & $C_2 \times \operatorname{A}_4$ & 3 & 3 & Center is non-trivial and  Prop. \ref{Prop.non-trivial-center}\\ 
 \hline
 (24,14) &$C_2 \times C_2 \times \operatorname{S}_3$ & 3 & 3 &  Center is non-trivial and  Prop. \ref{Prop.non-trivial-center} \\ 
 \hline
 (26,1) &  $D_{26}$ & 2 & 1 & Theorem~\ref{ED=1}(a) \\ 
 \hline
 (27,3) & $ (C_3 \times C_3 )\rtimes C_3$ & 3 & 3 & Theorem~\ref{p-groups}\\ 
 \hline
 (27,4) & $ C_9 \rtimes C_3$ & 3 & 3  & Theorem~\ref{p-groups}\\ 
 \hline
 (28,1) & $ C_7 \rtimes C_4$ & 2 & 2 & Theorem~\ref{ED=1}(b)  \\ 
 \hline
 (28,3) & $D_{28}$ & 2 & 2 & Theorem~\ref{ED=1}(b) \\ 
 \hline
 (30,1) & $C_5\times D_6$ & 2 & 2 & Theorem~\ref{ED=1}(b) \\ 
 \hline
 (30,2) & $C_3 \times D_{10}$ & 2 & 2 & Theorem~\ref{ED=1}(b) \\ 
 \hline 
 (30,3) & $D_{30}$ & 2 & 1 & Theorem~\ref{ED=1}(a) \\ 
 \hline
\end{tabular}
\end{center}

 \section{Non-abelian groups of order 32}
 \label{sect.32}
 
\begin{center}
\begin{tabular}{ |p{1.5cm}|p{4.1cm}|p{1.2cm}||p{1.5 cm}|p{3.6cm}|p{1.2cm}|} 
 \hline
 GAP Id & Structure Description & Rd=Ed & GAP Id & Structure Description & Rd=Ed \\ 
 \hline\hline
(32, 2) & $ (C_4 \times C_2) \rtimes C_4 $ & 4 &  (32, 4) & $C_8 \rtimes C_4 $ & 3\\
\hline
(32,5) & $ (C_8 \times C_2) \rtimes C_2 $  & 3 & (32, 6) & $ (C_2 \times C_2 \times C_2) \rtimes C_4 $ & 4 \\
\hline
(32, 7) & $ (C_8 \rtimes C_2) \rtimes C_2 $  & 4 & (32, 8) & $ (C_2 \times C_2) \cdot (C_4 \times C_2)$ & 4 \\
\hline
(32,9) & $ (C_8 \times C_2) \rtimes C_2 $   & 3 & (32, 10) & $ Q_8 \rtimes C_4 $ & 3\\
\hline
(32, 11) & $ (C_4 \times C_4) \rtimes C_2 $ & 2 &  (32, 12) &  $C_4 \rtimes C_8 $& 3 \\
\hline
(32,13) & $C_8 \rtimes C_4 $& 3 & (32, 14) & $C_8 \rtimes C_4 $ & 3\\
\hline
(32, 15) & $C_4 \cdot D_8 $ & 2 &  (32, 17) & $C_{16}\rtimes C_2$ & 2 \\
\hline
(32, 18) & $D_{32}$ & 2 & (32, 19) & $QD_{32}$ & 2\\
\hline
(32, 20) & $Q_{32}$ & 2 &(32, 22) & $C_2 \times ((C_4 \times C_2)\rtimes C_2)$&  4 \\
\hline
(32,23) & $C_2 \times (C_4 \rtimes C_4) $ & 4 &(32, 24) & $(C_4 \times C_4) \rtimes C_2$& 3\\
\hline
(32, 25) & $C_4 \times D_8$& 3 & (32, 26) & $C_4 \times Q_8$& 3 \\
\hline
(32,27) & $ (C_2 \times C_2 \times C_2\times C_2) \rtimes C_2 $ & 4 & (32, 28) & $ (C_4 \times C_2 \times C_2) \rtimes C_2 $& 4\\
\hline
(32, 29) & $ (C_2 \times Q_{16}) \rtimes C_2 $& 4 & (32, 30) &$ (C_4 \times C_2 \times C_2) \rtimes C_2 $ & 4\\
\hline
(32,31) & $ (C_4 \times C_4 ) \rtimes C_2 $ & 4 & (32, 32) & $ (C_2 \times C_2 ) \cdot (C_2 \times C_2 \times C_2) $& 4 \\
\hline
(32, 33) & $ (C_4 \times C_4 ) \rtimes C_2 $ & 4 & (32, 34) & $ (C_4 \times C_4 ) \rtimes C_2 $ & 4\\
\hline
(32,35) & $C_4 \rtimes Q_8$ & 4 & (32, 37) & $C_2 \times (C_8 \rtimes C_2)$& 3\\
\hline
(32, 38) & $(C_8 \times C_2) \rtimes C_2$ & 2 & (32, 39) & $C_2 \times D_{16}$& 3\\
\hline
(32,40) & $C_2 \times QD_{16}$& 3 & (32, 41) & $C_2 \times Q_{16}$& 3\\
\hline
(32, 42) & $(C_8 \times C_2)\rtimes C_2$ & 2 & (32, 43) &$C_8 \rtimes (C_2 \times C_2) $ & 4\\
\hline
(32,44) & $(C_2 \times Q_{16}) \rtimes C_2$& 4 &(32, 46) & $C_2 \times C_2 \times D_8$& 4\\
\hline
(32, 47) &$C_2 \times C_2 \times Q_8$ & 4 & (32, 48) & $C_2 \times ((C_4 \times C_2)\rtimes C_2)$ & 3\\
\hline
(32,49) & $ (C_2 \times C_2 \times C_2) \rtimes (C_2\times C_2) $ & 4 & (32, 50) & $(C_2 \times Q_2) \rtimes C_2$ & 4\\
\hline
 \end{tabular}
\end{center}

\section{Non-abelian groups of order 33 - 47}
\label{sect.33-47}

\begin{center}
\begin{tabular}{ |p{1.5cm}|p{4cm}|p{0.5cm}|p{0.5 cm}|p{7cm}| } 
 \hline
 GAP Id & Structure Description & Rd & Ed & Explanation \\ 
 \hline\hline
 (34,1) & $D_{34} $ &  2 & 1 & Theorem~\ref{ED=1}(a) \\ 
 \hline
 (36,1) & $C_9 \rtimes C_4$ & 2 & 2 & Theorem~\ref{ED=1}(b) \\ 
 \hline
 (36,3) & $(C_2 \times C_2) \rtimes C_9 $ & 3 & 3 &  Center is non-trivial and  Prop. \ref{Prop.non-trivial-center}\\ 
 \hline
 (36,4) & $D_{36}$ & 2 & 2 & Theorem~\ref{ED=1}(b) \\ 
 \hline
 (36,6) & $C_3 \times (C_3 \rtimes C_4) $ & 2 & 2 &  Theorem~\ref{ED=1}(b)  \\ 
 \hline
 (36,7) & $(C_3 \times C_3) \rtimes C_4$ & 4 & 3 & Proposition~\ref{prop.product}(a) with $A \simeq C_2$. Note that $G/A\simeq (C_3 \times C_3) \rtimes C_2$ has GAP Id (18, 4), trivial center, and essential dimension $2$.   \\ 
 \hline
 (36,9) & $(C_3 \times C_3) \rtimes C_4$ & 4 & 2 & Subgroup of $(C_3 \times C_3) \rtimes \mathcal{G}_2$ and Theorem~\ref{ED=2}(iii)\\ 
 \hline
 (36,10) &$D_{6}\times D_6$ & 4 & 2 & By Lemma~\ref{lem.first}(e), $\ed (D_{6}\times D_6) \leqslant \ed(D_6)+\ed(D_6)=2$   \\ 
 \hline
  (36,11) &$C_{3}\times \operatorname{A}_4$ & 3 & 3 & Center is non-trivial and  Prop. \ref{Prop.non-trivial-center}  \\ 
 \hline
  (36,12) &$C_{6}\times D_6$ & 2 & 2 &  Theorem~\ref{ED=1}(b) \\ 
 \hline
  (36,13) & $C_2 \times((C_3 \times C_3) \rtimes C_2)$ & 4 & 3 & Proposition~\ref{prop.product}(b) with $A = C_2$ and $H =(C_3 \times C_3) \rtimes C_2$. Here the GAP Id of $H$ is (18,4),
  $Z(H)$ is trivial, and $\ed(H) = 2$. \\ 
 \hline
(38,1) & $D_{38}$ & 2 & 1 & Theorem~\ref{ED=1}(a) \\ 
 \hline
 (39,1) & $C_{13} \rtimes C_3$ & 3 & 2 &  Example~\ref{ex.(39,1)} \\
 \hline
 (40,1) & $C_5 \rtimes C_8 $ & 2 & 2 & Theorem~\ref{ED=1}(b)  \\ 
 \hline
 (40,3) & $C_5 \rtimes C_8$ & 4 & 3 & Proposition~\ref{prop.product}(a) with $A \simeq C_2$. Note that $G/A \simeq C_5 \rtimes C_4$ has GAP number (20, 3), trivial center, and essential dimension $2$. \\ 
 \hline
 (40,4) & $C_5 \rtimes Q_8$ & 2 & 2 & Theorem~\ref{ED=1}(b)  \\ 
 \hline
 (40,5) &$C_{4}\times D_{10}$ & 2 & 2 & Theorem~\ref{ED=1}(b)  \\ 
 \hline
 (40,6) & $D_{40}$ & 2 & 2 & Theorem~\ref{ED=1}(b) \\ 
 \hline
 (40,7) & $C_2 \times( C_5 \rtimes C_4) $ & 3 & 3 &  Center is non-trivial and  Prop. \ref{Prop.non-trivial-center} \\ 
 \hline
 (40,8) & $ (C_{10}\times C_2) \rtimes C_2 $ & 2 & 2 & Theorem~\ref{ED=1}(b)  \\ 
 \hline
 (40,10) &$C_{5}\times D_{8}$ & 2 & 2 & Theorem~\ref{ED=1}(b)  \\ 
 \hline
 (40,11) &$C_{5}\times Q_{8}$ & 2 & 2 & Theorem~\ref{ED=1}(b)  \\ 
 \hline
 (40,12) & $C_2 \times( C_5 \rtimes C_4) $ & 4 & 3 &$G=C_2 \times H,
 $ where $Z(H)$ is trivial and GAP Id of $H$ is (20,3). By 
 Prop.~\ref{prop.product}(b), $\ed(G)=\ed(C_2 \times H)=\ed(H)+1=3.$  \\ 
 \hline
 (40,13) & $C_{2}\times C_2 \times D_{10}$ & 3 & 3 &Center is non-trivial and  Prop. \ref{Prop.non-trivial-center} \\ 
 \hline
 (42,1) & $C_7 \rtimes C_6 $ & 6 & 2 & $\ed(G) \geqslant 2$ by Theorem~\ref{ED=1}(a), and $\ed(G) \leqslant 2$ by
 Theorem~\ref{Th.Semidirect-product} with $q = 7$. \\ 
 \hline
 (42,2) & $ C_2 \times (C_7 \rtimes C_3) $ & 3 & 3 & Center is non-trivial and  Prop.~\ref{Prop.non-trivial-center} \\ 
 \hline
 (42,3) &$C_{7}\times D_{6}$ & 2 & 2 & Theorem~\ref{ED=1}(b) \\ 
 \hline
 (42,4) &$C_{3}\times D_{14}$ & 2 & 2 &  Theorem~\ref{ED=1}(b) \\ 
 \hline
 (42,5) & $ D_{42}$ & 2 & 1 &  Theorem~\ref{ED=1}(a) \\ 
 \hline
 (44,1) & $C_{11} \rtimes C_4 $ & 2 & 2 & Theorem~\ref{ED=1}(b)  \\ 
 \hline
 (44,3) & $D_{44}$ &  2 & 2 & Theorem~\ref{ED=1}(b) \\ 
 \hline
 (46,1) & $D_{46}$ & 2 & 1 & Theorem~\ref{ED=1}(a) \\ 
 \hline
\end{tabular}
\end{center}

\section{Non-abelian groups of order 48}
\label{sect.48}

\begin{center}
\begin{tabular}{ |p{1.5cm}|p{4cm}|p{0.5cm}|p{0.5 cm}|p{7cm}| } 
 \hline
 GAP Id & Structure Description & Rd & Ed & Explanation\\ 
 \hline\hline
 (48,1) & $C_{3} \rtimes C_{16}$ & 2 & 2 & Theorem~\ref{ED=1}(b) \\
 \hline
 (48,3) & $(C_4 \times C_4) \rtimes C_3$ & 3 & 2 & $\rd(G) \geqslant 2$ because $G$ contains $C_4 \times C_4$. $\rd(G) \leqslant 2$ because $G$ is a subgroup of $T \rtimes \smallskip \mathcal{G}_3$ from Theorem~\ref{ED=2}{\rm(iv)}, with $G \cap T = C_4 \times C_4$      \\
 \hline
 (48,4) &$C_{8}\times D_{6}$ & 2 & 2 & Theorem~\ref{ED=1}(b)  \\ 
 \hline
 (48,5) & $C_{24} \rtimes C_2$ & 2 & 2 & Theorem~\ref{ED=1}(b)  \\ 
 \hline
 (48,6) & $C_{24} \rtimes C_2$ &  2 & 2 & Theorem~\ref{ED=1}(b) \\ 
 \hline
 (48,7) & $D_{48}$ & 2 & 2 & Theorem~\ref{ED=1}(b) \\ 
 \hline
 (48,8) & $ C_3 \rtimes Q_{16}$ & 2 & 2 & Theorem~\ref{ED=1}(b) \\ 
 \hline
 (48,9) & $ C_2 \times (C_3 \rtimes C_8)$ & 3 & 3 & Center is non-trivial and  Prop.~\ref{Prop.non-trivial-center}\\
 \hline
 (48,10) & $(C_3 \rtimes C_8)\rtimes C_2 $ & 2 & 2  & Theorem~\ref{ED=1}(b) \\
 \hline
 (48,11) & $C_4 \times (C_3 \rtimes C_4)$ & 3 & 3 &  Center is non-trivial and  Prop.~\ref{Prop.non-trivial-center}\\ 
 \hline
 (48,12) & $(C_3 \rtimes C_4)\rtimes C_4$ & 3 & 3 &  Center is non-trivial and  Prop.~\ref{Prop.non-trivial-center}\\ 
 \hline
 (48,13) & $C_{12} \rtimes C_4 $ & 3 & 3 & Center is non-trivial and  Prop.~\ref{Prop.non-trivial-center}\\ 
 \hline
 (48,14) & $(C_{12} \times C_2)\rtimes C_2 $ & 3 & 3 & Center is non-trivial and  Prop.~\ref{Prop.non-trivial-center}\\ 
 \hline
 (48,15) & $(C_3 \times D_{8})\rtimes C_2$ & 4 & 3 & $\ed(G)\geqslant 3$ by  Prop.~\ref{Prop.non-trivial-center}   and  $\ed(G) \leqslant 3$ by Prop.~\ref{prop.48}, where $G_2 \simeq {G}/{G_3}\simeq D_{16}.$\\ 
 \hline
 (48,16) & $(C_3 \rtimes Q_{8})\rtimes C_2$ & 4 & 3 & $\ed(G)\geqslant 3$ by  Prop.~\ref{Prop.non-trivial-center} and $\ed(G) \leqslant 3$ by Prop.~\ref{prop.48}, where $G_2 \simeq {G}/{G_3}\simeq QD_{16}.$\\
 \hline
 (48,17) & $(C_3 \times Q_{8})\rtimes C_2$ & 4 & 3 & $\ed(G)\geqslant 3$ by  Prop.~\ref{Prop.non-trivial-center} and  $\ed(G) \leqslant 3$ by Prop.~\ref{prop.48}, where $G_2 \simeq {G}/{G_3}\simeq QD_{16}.$\\
 \hline
 (48,18) & $C_3 \rtimes Q_{16}$ & 4 & 3 & $\ed(G)\geqslant 3$ by  Prop.~\ref{Prop.non-trivial-center}  and $\ed(G) \leqslant 3$ by Prop.~\ref{prop.48}, where $G_2 \simeq {G}/{G_3}\simeq Q_{16}.$  \\
 \hline
 (48,19) & $(C_6 \times C_2)\rtimes C_4$ & 3 & 3 &Center is non-trivial and  Prop. \ref{Prop.non-trivial-center}\\ 
 \hline
 (48,21) & $ C_3 \times ((C_4 \times C_2)\rtimes C_2)$  & 3 & 3 &Center is non-trivial and  Prop. \ref{Prop.non-trivial-center}\\ 
 \hline
 (48,22) & $C_3 \times (C_4\rtimes C_4)$ & 3 & 3 &Center is non-trivial and Prop. \ref{Prop.non-trivial-center}\\ 
 \hline
 (48,24) & $C_3 \times (C_8\rtimes C_2)$ & 2 & 2 & Theorem~\ref{ED=1}(b) \\ 
 \hline
 (48,25) &$C_{3}\times D_{16}$ & 2 & 2 & Theorem~\ref{ED=1}(b) \\ 
 \hline
 (48,26) & $C_{3}\times QD_{16}$& 2  & 2 & Theorem~\ref{ED=1}(b) \\ 
 \hline 
 (48,27) &$C_{3}\times Q_{16}$  & 2 & 2 & Theorem~\ref{ED=1}(b) \\ 
 \hline
 (48,28) & $C_2 \cdot \operatorname{S}_4 $ & 2 & 2 & Theorem~\ref{ED=1}(b) \\ 
 \hline
 (48,29) & $GL_2(\mathbb{F}_3)$ & 2 & 2 &  Theorem~\ref{ED=1}(b) \\ 
 \hline
 (48,30) & $\operatorname{A}_4 \rtimes C_4$ & 3 & 3  &Center is non-trivial and Prop. \ref{Prop.non-trivial-center}\\
 \hline
 \end{tabular}
\end{center}

\section{Non-abelian groups of order 48, continued}
\label{sect.48-continued}

\begin{center}
\begin{tabular}{ |p{1.5cm}|p{4cm}|p{0.5cm}|p{0.5cm}|p{7cm}| } 
 \hline
 GAP Id & Structure Description & Rd & Ed & Explanation\\ 
 \hline\hline

 (48,31) & $\operatorname{A}_4 \times C_4 $ & 3 & 3  &Center is non-trivial and Prop. \ref{Prop.non-trivial-center}\\
 \hline
 (48,32) &$C_{2}\times SL_2(\mathbb{F}_3)$ & 3 & 3 & Center is non-trivial and Prop. \ref{Prop.non-trivial-center}\\ 
 \hline
 (48,33) & $SL_2(\mathbb{F}_3)\rtimes C_2$ & 2 & 2 &  Theorem~\ref{ED=1}(b)  \\ 
 \hline
 (48,34) & $C_2 \times (C_3 \rtimes Q_8)$ & 3 & 3 & Center is non-trivial and Prop.~\ref{Prop.non-trivial-center} \\ 
 \hline
 (48,35) &$C_{2}\times C_{4}\times \operatorname{S}_{3}$ & 3 & 3 & Center is non-trivial and Prop. \ref{Prop.non-trivial-center}\\ 
 \hline
 (48,36) &$C_{2}\times D_{24}$ & 3 & 3 & Center is non-trivial and Prop. \ref{Prop.non-trivial-center}\\ 
 \hline
 (48,37) & $(C_{12} \times C_2)\rtimes C_2 $ & 2 & 2 &  Theorem~\ref{ED=1}(b)  \\ 
 \hline
 (48,38) &$D_{8}\times \operatorname{S}_{3}$ & 4 & 3 & $\ed(G)\geqslant 3$ by Prop. \ref{Prop.non-trivial-center} and $\ed(G)\leqslant \ed(D_8)+\ed(\operatorname{S}_3)=3$ 
 by Lemma~\ref{lem.first}(e). \\
 \hline
 (48,39) & $ (C_4 \times \operatorname{S}_3) \rtimes C_2 $ & 4 & 3 & $\ed(G)\geqslant 3$ by  Prop.~\ref{Prop.non-trivial-center} and  $\ed(G) \leqslant 3$ by Prop.~\ref{prop.48}, where GAP id of $G_2 \simeq {G}/{G_3}$ is (16,13).\\
 \hline
 (48,40) &$\operatorname{S}_{3}\times Q_{8}$  & 4 & 3 & 
 $\ed(G) \geqslant 3$ by Prop.~\ref{Prop.non-trivial-center} and $\ed(G)\leqslant \ed(Q_8)+\ed(\operatorname{S}_3)=3$ by Lemma~\ref{lem.first}(e). \\
 \hline
 (48,41) & $(C_4 \times \operatorname{S}_3) \rtimes C_2 $ & 4 & 3 & $\ed(G)\geqslant 3$ by  Prop.~\ref{Prop.non-trivial-center} and  $\ed(G) \leqslant 3$ by Prop.~\ref{prop.48}, where GAP id of $G_2 \simeq {G}/{G_3}$ is (16,13).\\ 
 \hline
 (48,42) & $C_2 \times C_2 \times (C_3 \rtimes C_4) $ & 4 & 4 & Proposition~\ref{prop.product}(b) with $A = C_2 \times C_2$
 and $H = C_3 \rtimes C_4$. Here $H$ has GAP Id (12, 1), $Z(H) \simeq C_2$, and $\ed(H) = 2$. 
\\  
 \hline
 (48,43) & $ C_2 \times ((C_6 \times C_2) \rtimes C_2) $ & 3 & 3 &Center is non-trivial and Prop.\ref{Prop.non-trivial-center}\\ 
 \hline
 (48,45) &$C_{6}\times D_{8}$ & 3 & 3 &  Center is non-trivial and Prop. \ref{Prop.non-trivial-center}\\ 
 \hline
 (48,46) &$C_{6}\times Q_{8}$ & 3 & 3 & Center is non-trivial and Prop. \ref{Prop.non-trivial-center}\\
 \hline
 (48,47) & $C_3 \times ((C_4 \times C_2) \rtimes C_2)$ & 2 & 2  &  Theorem~\ref{ED=1}(b)  \\
 \hline
 (48,48) &$C_{2}\times \operatorname{S}_{4}$ & 3 & 3 &Center is non-trivial and Prop. \ref{Prop.non-trivial-center}\\ 
 \hline
 (48,49) &$C_{2}\times C_2\times \operatorname{A}_{4}$ & 4 & 4 & 

 Prop.~\ref{prop.product}(b) with $A = C_2 \times C_2$.
 
 \\ 
 \hline
 (48,50) & $(C_2 \times C_2 \times C_2 \times C_2) \rtimes C_3 $ & 6 & 4 & $C_2 \times C_2 \times C_2 \times C_2$ is subgroup of $G$ and $G$ is subgroup of $\operatorname{A}_4 \times \operatorname{A}_4$; see Lemma~\ref{lem.(48,50)}. Using Lemma \ref{lem.first}(c) and (e), $4=\ed(C_2 \times C_2 \times C_2 \times C_2)\leqslant \ed(G)\leqslant \ed(\operatorname{A}_4 \times \operatorname{A}_4 ) = 4.$  \\ 
 \hline
 (48,51) & $C_{2}\times C_2\times C_2 \times \operatorname{S}_{3}$ & 4 & 4 & Prop.~\ref{prop.product}(b) with $A = C_2 \times C_2 \times C_2$.  \\ 
 \hline
\end{tabular}
\end{center}

 \section{Non-abelian groups of order 50 - 55}
 \label{sect.50-55}
 
\begin{center}
\begin{tabular}{ |p{1.5cm}|p{4cm}|p{0.5cm}|p{0.7 cm}|p{7cm}| } 
 \hline
 GAP Id &Structure Description &  Rd & Ed & Explanation\\ 
 \hline\hline
 (50,1) & $D_{50} $ & 2 & 1 & Theorem~\ref{ED=1}(a) \\ 
 \hline
 (50,3) & $C_5 \times D_{10}$ & 2 & 2 &  Theorem~\ref{ED=1}(b) . \\ 
 \hline
 (50,4) & $(C_5 \times C_5)\rtimes C_2$ & 4 & 2 &  Subgroup of $(C_5 \times C_5) \rtimes \mathcal{G}_1$ and Theorem~\ref{ED=2}(ii). \\ 
 \hline
 (52,1) & $C_{13}\rtimes C_4 $ &  2 & 2 &  Theorem~\ref{ED=1}(b) .  \\ 
 \hline
 (52,3) & $C_{13} \rtimes C_4$ & 4 & 2 & Subgroup of $(C_{13} \times C_{13}) \rtimes \mathcal{G}_2$ and Theorem~\ref{ED=2}(iii). \\
 \hline
 (52,4) & $D_{52}$ & 2 & 2 &  Theorem~\ref{ED=1}(b) . \\ 
 \hline
 (54,1) & $D_{54}$ & 2 & 1 & Theorem~\ref{ED=1}(a) \\ 
 \hline
 (54,3) & $C_3 \times D_{18}$ & 2 & 2 &   Theorem~\ref{ED=1}(b) . \\ 
 \hline
 (54,4) & $C_9 \times \operatorname{S}_3 $ & 2 & 2 &  Theorem~\ref{ED=1}(b) .\\ 
 \hline
 (54,5) & $(C_3 \times C_3)\rtimes C_6 $ & 6 & 3-5 & $\ed(G) \geqslant 3$ because $G$ contains the non-abelian group $G_3 = (C_3 \times C_3)\rtimes C_3$ of order $27$, and $\ed(G) \geqslant \ed(G_3) = \rd(G_3) = 3$. On the other hand, $\ed(G) \leqslant 5$ by Lemma~\ref{lem.first}(a) because the center of $G$ is trivial. \\
 \hline
  (54,6) & $C_9\rtimes C_6 $ & 6 & 3 &  $G$ contains the non-abelian group $G_3 = C_9 \rtimes C_3$ of order $27$; hence, $\ed(G) \geqslant \ed(G_3) = \rd(G_3) = 3$.
  On the other hand, $G$ contains the dihedral group $D_{18} = C_9 \rtimes C_2$ as a subgroup of index $3$. By Lemma~\ref{ED=1}(a),
  $\ed(D_{18}) = 1$. Moreover, $\ed(G) \leqslant \ed(D_{18}) \cdot [G:D_{18}] = 1 \cdot 3 = 3$
  by Lemma~\ref{lem.first}(d). \\
 \hline
 (54,7) & $(C_9 \times C_3) \rtimes C_2$ & 4 & 2 & $G$ is subgroup of $D_6 \times D_{18}$ and using Proposition~\ref{prop.product}(b), $\ed(D_6 \times D_{18})\leqslant \ed(D_6)+\ed(D_{18})=2.$ Now Use Lemma  \ref{lem.first}(e) and Theorem~\ref{ED=1}(b).\\ 
 \hline
 (54,8) & $((C_3\times C_3)\rtimes C_3)\rtimes C_2$ & 3 & 3 &  Center is non-trivial and Prop.~\ref{Prop.non-trivial-center}. \\ 
 \hline
 (54,10) & $C_2 \times((C_3\times C_3)\rtimes C_3)$ & 3 & 3 & Center is non-trivial and Prop.~\ref{Prop.non-trivial-center}. \\ 
 \hline
 (54,11) & $C_2 \times (C_9 \rtimes C_3) $ & 3 & 3 &Center is non-trivial and Prop.~\ref{Prop.non-trivial-center}. \\ 
 \hline
 (54,12) & $C_{3}\times C_3\times \operatorname{S}_{3}$ & 3 & 3 &Center is non-trivial and Prop.~\ref{Prop.non-trivial-center}. \\ 
 \hline
 (54,13) & $C_3 \times ((C_3 \times C_3) \rtimes C_2) $ & 4 & 3 &$G=C_3 \times H,
 $ where $Z(H)$ is trivial and GAP Id of $H$ is (18,4). Using 
 Prop.~\ref{prop.product}(b), $\ed(G)=\ed(C_3 \times H)=\ed(H)+1=3.$\\ 
 \hline
 (54,14) & $ (C_3 \times C_3 \times C_3) \rtimes C_2 $ & 6 & 3 & $C_3 \times C_3 \times C_3$ is subgroup of $G$ and $G$ is subgroup of $\operatorname{S}_3 \times \operatorname{S}_3 \times \operatorname{S}_3.$ Using Lemma \ref{lem.first}(c), $3=\ed(C_3 \times C_3 \times C_3)\leqslant \ed(G)\leqslant \ed(\operatorname{S}_3 \times \operatorname{S}_3 \times \operatorname{S}_3) = 3.$  \\ 
 \hline
 (55,1) & $C_{11} \rtimes C_5 $ & 5 & 3-4 &  $\ed(G) \geqslant 3$ by Theorem~\ref{ED=2}. The inequality $\ed(G) \leqslant 4$ follows from Lemma~\ref{lem.first}(a) because the center of $G$ is trivial. \\
 \hline
 
 \end{tabular}
\end{center}

 \section{Non-abelian groups of order 56 -  63}
 \label{sect.56-63}
 
\begin{center}
\begin{tabular}{ |p{1.5cm}|p{4cm}|p{0.5cm}|p{0.5 cm}|p{7cm}| } 
 \hline
 GAP Id &Structure Description &  Rd & Ed & Explanation\\ 
 \hline\hline
 (56,1) & $C_7 \rtimes C_8 $ & 2 & 2 &   Theorem~\ref{ED=1}(b) \\ 
 \hline
 (56,3) & $ C_7 \rtimes Q_8 $ & 2 & 2 &  Theorem~\ref{ED=1}(b)  \\ 
 \hline
 (56,4) & $C_4 \times D_{14}$ & 2 & 2 &  Theorem~\ref{ED=1}(b) \\ 
 \hline
 (56,5) & $D_{56}$ & 2 & 2 &  Theorem~\ref{ED=1}(b)  \\ 
 \hline
 (56,6) & $C_2 \times (C_7 \rtimes C_4) $  & 3 & 3 & Center is non-trivial and Prop.~\ref{Prop.non-trivial-center}\\ 
 \hline
 (56,7) & $(C_{14} \times C_2) \rtimes C_2 $ & 2 & 2 &   Theorem~\ref{ED=1}(b) \\ 
 \hline
 (56,9) & $C_{7}\times D_{8}$& 2 & 2 &  Theorem~\ref{ED=1}(b) . \\ 
 \hline
 (56,10) & $C_{7}\times Q_{8}$& 2 & 2 &  Theorem~\ref{ED=1}(b) \\ 
 \hline
 (56,11) & $(C_2 \times C_2\times C_2 )\rtimes C_7 $ & 7 & 3-6 & $\ed(G) \geqslant 3$ because $G$ contains $C_2\times C_2 \times C_2$. On the other hand, $\ed(G) \leqslant 6$ by Lemma~\ref{lem.first}(a) because the center of $G$ is trivial. \\ 
 \hline
 (56,12) & $C_{2}\times C_2 \times D_{14}$& 3 & 3 & Center is non-trivial and Prop. \ref{Prop.non-trivial-center}\\ 
 \hline
  (57,1) & $C_{19}\rtimes C_3 $ & 3 & 2 &Subgroup of $(C_{19} \times C_{19}) \rtimes \mathcal{G}_4$ and Theorem~\ref{ED=2}(v)\\ 
 \hline
 (58,1) & $D_{58}$ & 2 & 1 & Theorem~\ref{ED=1}(a)  \\ 
 \hline
 (60,1) & $C_5 \times (C_3 \rtimes C_4) $ & 2 & 2 &  Theorem~\ref{ED=1}(b)  \\ 
 \hline
 (60,2) & $ C_3 \times (C_5 \rtimes C_4)$ &  2 & 2 &   Theorem~\ref{ED=1}(b) \\ 
 \hline
 (60,3) & $C_{15}\rtimes C_4 $ & 2 & 2 &   Theorem~\ref{ED=1}(b) \\ 
 \hline
 (60,5) & $\operatorname{A}_5$ & 3 & 2 & Subgroup of $\operatorname{S}_5$ and Theorem~\ref{ED=2}(vii)\\ 
 \hline
 (60,6) & $ C_3 \times (C_5 \rtimes C_4)$ & 4 & 3 & Proposition~\ref{prop.product}(b) with $A=C_3$
 and $H= (C_5 \rtimes C_4)$. Here $Z(H)$ is trivial, the GAP Id of $H$ is (20,3), and $\ed(H) = 2$.\\ 
 \hline
 (60,7) & $ C_{15}\rtimes C_4$ & 4 & 3  &  $\ed(G) \leqslant 3$ by Lemma~\ref{lem.first}(a). $\ed(G) \geqslant 3$ by
 Example~\ref{ex.(60,7)}. \\ 
 \hline
 (60,8) & $\operatorname{S}_{3}\times D_{10}$& 4 & 2 & Lemma~\ref{lem.first}(e) and  Theorem~\ref{ED=1}(b)  \\ 
 \hline
 (60,9) &$C_{5}\times \operatorname{A}_{4}$ & 3 & 3 &  Center is non-trivial and Prop.  \ref{Prop.non-trivial-center} \\ 
 \hline
 (60,10) &$C_{6}\times D_{10}$ & 2 & 2 &   Theorem~\ref{ED=1}(b)  \\ 
 \hline
 (60,11) &$S_{3}\times C_{10}$ & 2 & 2 &   Theorem~\ref{ED=1}(b)  \\ 
 \hline
 (60,12) & $D_{60}$ & 2 & 2 &   Theorem~\ref{ED=1}(b) \\ 
 \hline
 (62,1) & $D_{62}$ & 2 & 1 &  Theorem~\ref{ED=1}(a) \\ 
 \hline
 (63,1) & $C_7 \rtimes C_9 $ & 3 & 3 & Center is non-trivial and Prop.  \ref{Prop.non-trivial-center}\\ 
 \hline
 (63,3) & $C_3 \times (C_7 \rtimes C_3) $ & 3 & 3 & Center is non-trivial and Prop.  \ref{Prop.non-trivial-center}\\ 
 \hline
\end{tabular}
\end{center}

\section{Preliminaries on essential dimension}
\label{sect.prel}

\begin{lemma} \label{lem.first}
{\rm(a)} $\ed(G) \leqslant \rd(G)$. Moreover, if the center of $G$ is trivial, then $\ed(G) \leqslant \rd(G) - 1$.

\smallskip
{\rm (b)} If $G$ is a finite abelian group, then $\ed(G) = \rd(G) = \operatorname{rank}(G)$. Here $\operatorname{rank}(G)$ denotes the rank of the finite abelian group $G$, i.e., the minimal number of generators of $G$.

\smallskip
{\rm (c)} Let $H$ be a subgroup of group $G$. Then $\ed(H) \leqslant \ed(G)$.

\smallskip
{\rm (d)}  Let $H$ be a subgroup of group $G$. Then $\ed(G) \leqslant [G: H] \cdot \ed(H)$.

\smallskip
{\rm (e)} $\ed(G_1 \times G_2) \leqslant \ed(G_1) +\ed(G_2)$.
\end{lemma}

\begin{proof} (a) Let $V$ be a faithful linear representation of $G$ of minimal dimension, $\dim(V) = \rd(G)$. Then $V$ itself is linearizable; hence, $\ed(G) \leqslant \dim(V) = \rd(G)$. Furthemore, if $G$ has trivial center, then the $G$-action on
the projective space $\mathbb P(V)$ is also faithful. Hence, $\mathbb P(V)$ is also linearizable and 
\[ \ed(G) \leqslant \dim \, \mathbb P(V) = \dim(V)  -1 = \rd(G) -1 \, ; \]
cf.~\cite[Corollary 6.18]{berhuy-favi}.

\smallskip
(b) is \cite[Theorem 6.1]{B-R}.

\smallskip
(c) Suppose $G$ acts faithfully on a variety $X$, and $X$ is linearizable as a $G$-variety of minimal dimension, $\dim(X) = \ed(G)$. Denote the linearization by $\pi \colon V \dasharrow X$, where $V$ is a faithful linear representation of $G$.
Restricting the $G$-action on $V$ and $X$ to the subgroup $H$, we may view $\pi$ as a linearization of the $H$-action on $X$. Thus $\ed(H) \leqslant \dim(X) = \ed(G)$, as claimed.

\smallskip
(d) If $V$ is a vector space with a linear action of $H$, and $V{\uparrow}_H^G$ is the induced representation of $G$. 
If $X$ is a $H$-variety, we can define the induced $G$-variety by $X{\uparrow}_H^G$ in the same way as $V{\uparrow}_H^G$, i.e., as follows.

The underlying variety of $X{\uparrow}_H^G$ is the direct product of $n = [G:H]$ copies of $X$, labeled by the left cosets of $H$ in $G$. Choosing representatives $g_1, \ldots, g_n$ for these cosets, 
we can write elements of $X{\uparrow}_H^G$ as $(x_1, \ldots, x_n)$. For any $g \in G$, write $g^{-1} g_i$ as
$g_j h_i$, for a suitable $j(i)  = 1, \ldots, n$ and $h_i \in H$. Then 
\[ g \cdot (x_1, \ldots, x_n) = (x_1', \ldots, x_n'), \]
where $x_i' = h_i^{-1}(x_{j(i)})$. The $G$-variety $X{\uparrow}_H^G$ produced by this construction
is independent of the choice of representatives $g_1, \ldots, g_n$ for the cosets of $H$ in $G$, up to (a $G$-equivariant) isomorphism.

The following two properties of
$X{\uparrow}_H^G$ are clear from this definition:

\smallskip
(i) If the $H$-action on $X$ is faithful, then the
$G$-action on $X{\uparrow}_H^G$ is faithful,

\smallskip
(ii) An $H$-equivariant morphism (respectively, rational map) $f \colon X \to Y$ induces a $G$-equivariant morphism (respectively, rational map) $f{\uparrow}_H^G \colon X{\uparrow}_H^G \to Y{\uparrow}_H^G$ given by
$$(x_1, \ldots, x_n) \mapsto (f(x_1), \ldots, f(x_n)).$$
To finish the proof, suppose $\ed(H) = d$.
Then there exists a faithful linearizable $H$-variety $X$ of dimension $d$. That is, there exists an $H$-equivariant dominant rational $V \dasharrow X$, where
$V$ is a linear representation of $H$. Then by property
(ii), there exists a dominant rational $G$-equivariant map $V{\uparrow}_H^G \mapsto X{\uparrow}_H^G$. Here
$V{\uparrow}_H^G$ is a linear representation of $G$ and
the $G$-action on $X{\uparrow}_H^G$
is faithful by (i). We conclude that 
\[ \ed(G) \leqslant \dim \, X{\uparrow}_H^G = n \cdot \dim(X) = [G:H] \cdot d = [G:H] \cdot \ed(H), \]
as claimed.

\smallskip
(e) If $X$ is a linearizable $G$-variety and $Y$ is a linearizble $H$-variety, then
$X \times Y$ is a linearizable $(G \times H)$-variety. Thus $\ed(G \times H) \leqslant \dim(X \times Y) = \dim(X) + \dim(Y)$.
Taking $X$ and $Y$ to be of minimal possible dimensions, $\ed(G)$ and $\ed(H)$, respectively, we see that
$\ed(G \times H) \leqslant \ed(G) + \ed(H)$.
\end{proof}

\begin{proposition} \label{prop.product} {\rm (R. L\"otscher)}

(a) Let $A$ be a central subgroup of a finite group $G$ such that $A \cap [G, G] = 1$. Then
\[ \ed(G) = \ed(G/A) - \rank(Z(G/A)) + \rank(Z(G)). \]

(b) In particular, 
$\ed(A \times H) = \ed(H) + \rank(Z(H) \times A) - \rank(Z(H))$
for any finite abelian group $A$ and any finite group $H$.
\end{proposition}

\begin{proof} Part (a) is Theorem 9 in~\cite{lotscher2010multihomogeneous}. (b) is a special case of (a); 
see~\cite[Corollary10]{lotscher2010multihomogeneous}. 
\end{proof} 

\begin{theorem} {\rm(\cite[Theorem 4.1]{km2}} \label{p-groups} Let $G$ be a finite $p$-group. Then $\ed(G) = \rd(G)$.
\end{theorem}

\begin{theorem}\label{Th.Semidirect-product}
   Let $q = p^n$ be a prime power. Then 
   $$\displaystyle \ed\big(  C_q \rtimes C_q^* \big)
\leqslant \phi(p-1)p^{n-1},$$ where $\phi$ is the Euler $\phi$-function. 
\end{theorem}

\begin{proof}
The result follows from the main theorem of \cite{Ledet-semidirect-product} and the fact that essential dimension of a finite group over $\mathbb C$ is less than or equal to same over $\mathbb Q$.
\end{proof}

\begin{corollary} \label{cor.ratio}
For positive real number $\epsilon$ there exists a finite group $G$ such that 
\[ \frac{\ed(G)}{\rd(G)} < \epsilon. \]
\end{corollary}

\begin{proof} Let $a = p_1 \cdot p_2 \ldots \cdot p_r$ be the product of the first $r$ primes. By Dirichlet's Theorem there exists a prime number $q$ of the form $q = a x + 1$, where $x$ is an integer. Now let $G = C_q \rtimes C_q^*$.
Then it is easy to see that $\rd(G) = q- 1$. On the other hand, by Theorem~\ref{Th.Semidirect-product},
\[ \ed(G) \leqslant \phi(q-1) \leqslant \left(1 - \frac{1}{p_1}\right)\left(1 - \frac{1}{p_2}\right) \ldots \left(1 - \frac{1}{p_r}\right) (q-1), \]
where the last inequality follows from the fact that $q-1 = ax$ is divisible by the primes $p_1, \ldots, p_r$. Now observe that 
\[ \left(1 - \frac{1}{p_1}\right)\left(1 - \frac{1}{p_2}\right) \ldots \left(1 - \frac{1}{p_r}\right) \to 0 \]
as $r \to \infty$. Thus
choosing $r$ sufficiently large, we will achieve
\[ \frac{\ed(G)}{\rd(G)} \leqslant \
\left(1 - \frac{1}{p_1}\right)\left(1 - \frac{1}{p_2}\right) \ldots \left(1 - \frac{1}{p_r}\right) < \epsilon , \]
as desired.
\end{proof}

\section{Groups of small essential dimension}
\label{sect.small-ed}

\begin{theorem} \label{ED=1}
(a) A finite group has essential dimension $1$ if and only if it is either cyclic group or dihedral group of order $2n$ where $n$ is an odd number. 
  
(b) Let $G$ be a non-abelian finite group of representation dimension $2$. Then either (i) $G$ is isomorphic to the dihedral group $D_{2n},$ with odd $n$ or $\ed(G) = 2$.
\end{theorem}

\begin{proof}
(a) is \cite[Theorem~6.2]{B-R}. (b) is an immediate consequence of (a) and the inequality
$\ed(G) \leqslant \rd(G)$; see~\eqref{e.ed-rd}.
\end{proof}

Finite groups of essential dimension $2$ has been classified by Duncan in \cite{Duncan-ed2}. We recall this classification here for convenience the  of the reader.

\begin{theorem}\cite[Theorem 1.1]{Duncan-ed2}\label{ED=2}
Let $T=(\mathbb{C}^{*})^2$ be a two dimensional torus. If $G$ is a finite group of essential dimension $2$ then $G$ is isomorphic to a subgroup of one the following groups:
\begin{enumerate}[{\rm (i)}]
    \item $GL_2(\mathbb{C}),$ the general linear groups of degree $2$.
    \item $T \rtimes \mathcal{G}_1 $ with $|G\cap T|$ coprime to $2$ and $3,$ and 
$\displaystyle \mathcal{G}_1=\left\langle \begin{pmatrix}1 & -1\\ 1 & 0 \end{pmatrix}, \begin{pmatrix} 0 & 1 \\ 1 & 0 \end{pmatrix} \right\rangle \simeq D_{12} $.  
\item $T \rtimes \mathcal{G}_2 $ with $|G\cap T|$ coprime to $2,$  and 
$\displaystyle \mathcal{G}_2=\left\langle \begin{pmatrix}-1 & 0\\ 0 & 1 \end{pmatrix}, \begin{pmatrix} 0 & 1 \\ 1 & 0 \end{pmatrix} \right\rangle \simeq D_{8}$. 
\item $T \rtimes \mathcal{G}_3 $ with $|G\cap T|$ coprime to $3,$ and 
$\displaystyle \mathcal{G}_3=\left\langle \begin{pmatrix}0 & -1\\ 1 & -1 \end{pmatrix}, \begin{pmatrix} 0 & -1 \\ -1 & 0 \end{pmatrix} \right\rangle \simeq \operatorname{S}_{3}$. 
\item $T \rtimes \mathcal{G}_4 $ with $|G\cap T|$ coprime to $3$, and 
$\displaystyle \mathcal{G}_4=\left\langle \begin{pmatrix}0 & -1\\ 1 & -1 \end{pmatrix}, \begin{pmatrix} 0 & 1 \\ 1 & 0 \end{pmatrix} \right\rangle \simeq \operatorname{S}_{3}$. 

\smallskip
\item $PSL_2(\mathbb{F}_7)$, the simple group of order $168$.

\smallskip
\item $\operatorname{S}_5$, the symmetric group of degree $5$.

\medskip
Furthermore, any finite subgroup of these groups has essential dimension $\leqslant 2$.
\end{enumerate}
\end{theorem}

\begin{example} \label{ex.(39,1)} Let $G$ be the semidirect product of $C_{13} = \langle a \rangle$ and $C_3= \langle b \rangle$, where $b$ acts on $C_{13}$ by $a \to a^3$. {\rm (}The Gap Id of this group is $(39,1)$.{\rm)} Then $\ed(G) =2$.

\smallskip
Indeed, $G$ is neither cyclic nor odd dihedral, so $\ed(G) \geqslant 2$. It remains to prove that $\ed(G) \leqslant 2$. We will do this by showing that $G$ is subgroup of $T \rtimes \mathcal{G}_3 $ from Theorem ~\ref{ED=2}(iv).

\smallskip
Let $T[13] \simeq C_{13} \times C_{13}$ 
be the 13-torsion subgroup of $T=(\mathbb{C}^{*})^2$. Let 
$m :=  \begin{pmatrix}0 & -1\\ 1 & -1 \end{pmatrix}$ be an element of order $3$ in the group
$\mathcal{G}_3 \simeq \operatorname{S}_{3}$ from Theorem ~\ref{ED=2}(iv). Consider the subgroup  $H := T[13]\rtimes C_3$ of $T \rtimes \mathcal{G}_3 $, where
$C_3 = \langle m \rangle$. Then $H$ has the following presentation:
$$H = \langle a_1, a_2, m ~|~ a_1^{13}=a_2^{13}=m^3=1, a_1a_2=a_2a_1, ma_1m^{-1}=a_2^{-1}, ma_2m^{-1}=a_1a_2^{-1} \rangle.$$

\smallskip
Now consider the subgroup $K=\langle a_1a_2^3, m \rangle$ of $H$ and note that $m$ acts on $a_1a_2^3$ by sending it to $(a_1a_2^3)^3$. This implies that the 
homomorphism
$G = \langle a, b \rangle \to H$ taking 
$a$ to $a_1 a_2^3$ and $b$ to $m$ is 
well defined and is an isomorphism between $G$ and $K$.  By Theorem~\ref{ED=2}, $\ed(G)= \ed(K) \leqslant 2$, as claimed.
\end{example} 

\begin{example} \label{ex.(60,7)} Let $G$ be the semidirect product of $C_{15} = \langle a \rangle$ and $C_4 = \langle b \rangle$, where $b$ acts on $C_{15}$ by $a \to a^2$. {\rm (}The Gap Id of this group is $(60,7)$.{\rm)} Then $\ed(G) \geqslant 3$.

\smallskip
Clearly $G$ is not cyclic or odd dihedral, so $\ed(G) \geqslant 2$. It thus remains to show that $G$ is not a subgroup of
any of the groups listed in Theorem~\ref{ED=2}.

Case (i) is ruled out as $\rd(G)=4$. 

Case (ii). Suppose that $G$ embeds into $T \rtimes \mathcal{G}_1$ so that
the order of $G \cap T$ is coprime to $2$ and $3$. Then $G_0 = G \cap T$ has order $1$ or $5$, and $G/G_0$ is isomorphic to a subgroup of $D_{12}$. If $|G_0| = 1$, then $|G/G_0| = |G| = 60$, so clearly $G/G_0$ cannot be isomorphic to a subgroup of $D_{12}$. If $|G_0| = 5$, i.e., $G_0 = \langle a^3 \rangle$,
then $G/G_0 \simeq C_3 \rtimes C_4$ has order $12$ and hence, must be isomorphic to $D_{12}$. This is easily seen not be the case. In particular,
the commutator is isomorphic to $C_3$ in both cases but modulo the commutator,
$C_3 \rtimes C_4$ is isomorphic to $C_4$, where as $D_{12}$ modulo the commutator is isomorphic to $C_2 \times C_2$. 

Case (iii). Assume the contrary: $G$ embeds into $T \rtimes \mathcal{G}_1$
so that $G \cap T$ is coprime to $2$. Then modulo $T$, $b$ projects to an element of order $4$ in $\displaystyle \mathcal{G}_2=\left\langle \begin{pmatrix}-1 & 0\\ 0 & 1 \end{pmatrix}, \begin{pmatrix} 0 & 1 \\ 1 & 0 \end{pmatrix} \right\rangle \simeq D_{8}$. 
Hence, $b^2$ projects to the central element $\displaystyle \begin{pmatrix}-1 & 0\\ 0 & -1 \end{pmatrix}$ of 
$\mathcal{G}_2$. In other words, conjugation by $b^2$ sends 
every $t \in T$ to $t^{-1}$. Consequently, $b^2$ does not commute with any element of $T$ whose order is not $2$. On the other hand, $b^2$ commutes with $a^5$, whose order is $3$, a contradiction.

Cases (iv) and (v). Once again, assume the contrary: $G$ embeds into $T \rtimes \Sym_3$ so that $G \cap T$ is coprime to $3$. Since $|\Sym_3|$ is not divisible by $5$ or $4$, $G \cap T$ contains the element $a^3$ of order $5$ and the element $b^2$ of order $2$. (Note that a priori, $b$ may project to an element of order $2$ in $\Sym_3$, but $b^2$ has to project to the identity element.) This is impossible because $a^3$ and $b^2$ do not commute in $G$.

Case (vi) is ruled out because $168$ is not divisible by $60$.
 
Case (vii) is ruled out as $\Sym_5$ has only one subgroup of order $60$.
This subgroup is $\Alt_5$, which is not isomorphic to $G$.
\qed
\end{example}

\begin{proposition}\label{Prop.non-trivial-center}
Let $G$  be a finite group with non-trivial center.
\begin{enumerate}
    \item If $\rd(G)=3$ then $\ed(G)=3$.
    \item If $\rd(G)\geqslant 3$ then $\ed(G) \geqslant 3$. 
\end{enumerate}
\end{proposition}
\begin{proof} We only need to prove (b). Since $\ed(G) \leqslant \rd(G)$, part (a) will follow.

To prove (b), we argue by contradiction. Assume that $\ed(G) \leqslant 2$. Our goal is to show that $\rd(G) \leqslant 2$.

If $\ed(G) = 1$ then by Theorem~\ref{ED=1}(a), 
$G$ is either cyclic or odd dihedral. Since the center of any odd dihedral group is trivial, we conclude that $G$ has to be cyclic. Then $\rd(G) = 1$.

Now assume $\ed(G) = 2$. In this case~\cite[Prop. 2.5]{Duncan-ed2}, tells us that $G$ embeds into $\GL_2(\mathbb C)$. In other words, $\rd(G) \leqslant 2$, as claimed.
\end{proof}

\section{Groups of order $48$}
\label{sect.prel48}

\begin{lemma} \label{lem.(48,50)}
Let $G$ be the semidirect product of $\mathbb{F}_2^4 \rtimes C_3$, where
a generatior of $C_3$ acts on $\mathbb{F}_2^4$ by the matrix 
\begin{equation} \label{e.4x4}
\begin{pmatrix} 0 & 1 & 0 & 0 \\ 1 & 1 & 0 & 0 \\ 0 & 0 & 0 & 1 \\ 0 & 0 & 1 & 1 \end{pmatrix}. 
\end{equation}
Here we identify the $4$-dimensional vector space $\mathbb F_2^4$ with the
elementary abelian group $C_2 \times C_2 \times C_2 \times C_2$. The group $G$ defined this way has GAP Id $(48,50)$. Then $G$ embeds into $\Alt_4 \times \Alt_4$, where $\Alt_4$ is the alternating group on four letters. 
\end{lemma}  

\begin{proof} The $4 \times 4$ matrix given by~\eqref{e.4x4} is block-diagonal with two identical diagonal $2 \times 2$ blocks,
\begin{equation} \label{e.2x2}
\begin{pmatrix} 0 & 1 \\ 1 & 1 \end{pmatrix}.
\end{equation}
Hence, $G$ embeds into the product of two copies 
of $H = \mathbb F_2^2 \rtimes C_3$ via
\[ ((a_1, a_2, a_3, a_3), \, c) \mapsto \big( (a_1, a_2), \, c), \; (a_3, a_4), c) \big) , \]
for any $a_1, a_2, a_3, a_4 \in \mathbb F_2$ and any $c \in C_3$. Here the
a generator of $C_3$ acts on $\mathbb F_2^2$ via the matrix~\eqref{e.2x2}.

It remains to show that $H := \mathbb F_2^2 \rtimes C_3$ is isomorphic to the alternating group $\Alt_4$. Indeed, $H$ acts on the four elements of $\mathbb F_2^2$ as follows: $\mathbb F_2^2$ acts by translations (the regular action) and
$C_3$ acts by $\mathbb F_2$-linear transformations as in~\eqref{e.2x2}. This 
gives rise to a group homomorphism $H \to \Sym_4$, whose image is readily seen to be $\Alt_4$.
\end{proof}

\begin{proposition} \label{prop.48}
Let $G$ be a group of order $48$, and $G_2$ and $G_3$ be $2$-Sylow and $3$-Sylow subgroups of $G$, respectively.
Assume $G_3$ is normal in $G$. Then

\smallskip
(a) either (i) $G \simeq G_2 \times G_3$ or (ii) there exists a surjective homomorphism $f \colon G \to \Sym_3$.

\smallskip
(b) There exists an injective group homomorphism $G \hookrightarrow G/G_3 \times \Sym_3 \simeq G_2 \times \Sym_3$.

\smallskip
(c) $\ed(G) \leqslant \ed(G_2) + 1$.
\end{proposition}

\begin{proof} (a) If $G_2$ is normal, then (i) holds. If $G_2$ is not normal, then by Sylow's Theorem,
there exist exactly three $2$-Sylow subgroups. Denote them by $G_2$, $G_2'$ and $G_2''$. Moreover, 
each of these $2$-Sylow subgroups is self-normalizing. The conjugation action of $G$ on the $2$-Sylow subgroups
give rise to a group homomorphism $f \colon G \to \Sym_3$. The kernel $K$ of $f$ is the intersection $G_2 \cap G_2' \cap G_2''$.
Since each of the groups $G_2, G_2'$ and $G_2''$ has order $16$, and they are distinct, we have
$|K| \leqslant 8$. On the other hand, 
\[ |K| \geqslant \frac{|G|}{|\Sym_3|} = \frac{48}{6} = 8 \, . \]
We conclude that $|K| = 8$ and $f$ is surjective. In other word, (ii) holds. Now observe that (i) and (ii) are incompatible:
in case (i), $G_3 \simeq C_3$ is central, whereas in case (ii) $G_3$ cannot be central, since $f(G_3) = \Alt_3$
is not central in $\Sym_3$.

\smallskip
(b) In case (i), part (b) is obvious. In case (ii), consider the homomorphism $(\pi \times f) \colon G \to G/G_3 \times \Sym_3$, where $\pi : G \to G/G_3$ is canonical epimorphism.
Note that homomorphism $(\pi \times f)$ is injective. Indeed, the kernel $K$ of $f \colon G \to \Sym_3$ has order $8$ and the kernel of $\pi \colon G \to G/G_3$ is of order $3$. Thus the intersection of the two kernels is trivial, and $\pi \times f$ is, indeed, injective. Finally observe that $\pi$ induces an isomorphism between $G_2$ and $G/G_3$.

\smallskip
(c) $\ed(G) \leqslant \ed(G_2 \times \Sym_3) \leqslant \ed(G_2) + \ed(\Sym_3) = \ed(G_2) + 1$.
\end{proof}

\bibliographystyle{amsalpha}
\bibliography{ED}

\end{document}